\documentclass[12pt]{amsart}

\usepackage{amsmath,amssymb,amscd,verbatim}

\newcommand{\sub}{\subseteq}

\newcommand{\ra}{\Rightarrow}

\newcommand{\al}{\alpha}

\newcommand{\be}{\beta}

\DeclareMathOperator{\Max}{Max}

\newtheorem{theorem}{Theorem}

\newtheorem{lemma}[theorem]{Lemma}

\newtheorem{proposition}[theorem]{Proposition}

\newtheorem{corollary}[theorem]{Corollary}

\newtheorem{Qu}{Problem}

\theoremstyle{definition}

\begin{document}

  \title{Locally principal ideals and finite character}


\keywords{}

\author{Stefania Gabelli}

\address{Dipartimento di Matematica, Universit\`{a} degli Studi Roma
Tre,
Largo S.  L.  Murialdo,
1, 00146 Roma, Italy
}

\email{gabelli@mat.uniroma3.it}

\date{}



\begin{abstract} It is well-known that if $R$ is a domain with finite character,  each locally principal nonzero  ideal of  $R$ is invertible. We address the problem of understanding when the converse is true and survey some recent results.
 \end{abstract}


\maketitle

\section*{}

 Throughout all the paper, $R$ will be an integral domain and $K$ its field of fractions. To avoid trivialities, we will assume that $R\neq K$. 
  
 An $R$-submodule $M$ of $K$ is called \emph{fractional} if $(R:M):=\{x\in K;\, xM\sub R\} \neq (0)$, equivalently, if there exists a nonzero $d\in R$ such that $dM\sub R$. If there is such a $d$, then $dM:=J$ is an ideal of $R$ and $M=d^{-1}J$; thus $M$ is also called  a \emph{fractional ideal}. 
 
  If $I$ is a fractional ideal of $R$, we call $I$ simply an \emph{ideal}  and if $I\sub R$ we say that $I$ is an \emph{integral ideal}. 
  
  An \emph{overring} of $R$ is a domain $D$ such that $R\sub D\sub K$. It is easy to see that $D$ is a fractional overring if and only if $D$ is the endomorphism ring of a nonzero integral ideal $I$ of $R$, that is, $D=E(I):=(I:I):=\{x\in K;\, xI\sub I\}$. In fact, if $I\sub R$ is a nonzero ideal, $E(I):=(I:I)$ is an overring and $(0) \neq I\sub (R:E(I))$. Conversely, if $D$ is a fractional overring and $I:=dD\sub R$,  $d\neq 0$, we have that $I$ is an ideal and $D=E(I):=(I:I)$.

An ideal $I$ of $R$ is called \emph{locally principal} if $I_M:=IR_M$ is principal, for each maximal ideal $M$ of $R$. Note that, if $I$ is a locally principal nonzero ideal, its endomorphism ring $E(I):=(I:I)$ is equal to $R$. In fact $(I_M:I_M)=R_M$, for each $M\in \Max(R)$. Hence, 
$$E(I)=(I:I)=\bigcap_{M\in \Max(R)} (I:I)_M=\bigcap_{M\in \Max(R)} (I_M:I_M)= \bigcap_{M\in \Max(R)} R_M =R.$$

  By checking it locally, it is also easy to see that a locally principal nonzero  ideal is  \emph{cancellative}, that is, whenever $IB = IC$ for ideals $B$ and $C$ of $R$, then $B = C$ \cite[Section 6]{gilmer}. 
  As proved by D. D. Anderson and M. Roitman, the converse is also true. 
  
 \begin{proposition} \cite[Theorem]{AR} A nonzero ideal $I$ of $R$ is cancellative if and only if it is locally principal. \end{proposition} 
 
 The set of nonzero ideals of $R$ is a multiplicative semigroup,  with unity $R$.  A nonzero ideal $I$ is \emph{invertible} if $IJ=R$, for some ideal $J$. If $I$ is invertible, its inverse is $(R:I):=\{x\in K; \,xI\sub R\}$ \cite[Section 7]{gilmer}.  
Clearly an invertible ideal is cancellative, hence locally principal. 

For a converse, we need that $I$ be finitely generated.
In fact, if $IJ=R$ and $1=a_1b_1+ \dots + a_nb_n$, $a_i\in I, b_i\in J$, the ideal $I$ is generated by $a_1, \dots, a_n$. For a discussion on the minimal number of generators of an invertible ideal one can see \cite{OR}.

\begin{proposition} \cite[Corollary 7.5]{gilmer} \label{inv} A nonzero ideal $I$ of $R$ is invertible if and only if it is finitely generated and locally principal.
\end{proposition} 

Thus, a nonzero finitely generated  ideal is invertible if and only if is locally principal, if and only if is cancellative. In particular, locally principal nonzero  ideals of Noetherian domains are invertible.

However, a locally principal ideal need not be finitely generated. In fact there are examples of domains such that $R_M$ is a one-dimensional  discrete valuation domain, for each $M\in \Max(R)$ (these domains are called \emph{almost Dedekind domains}), yet some of the maximal ideals are not finitely generated \cite[Section 36]{gilmer}. 

Thus one can ask:

\begin{Qu} \label{Q1} When is a locally principal nonzero ideal invertible (equivalently finitely generated)? \end{Qu}

A first answer is the following.

\begin{proposition} \cite[Theorem 3]{AZ}  \label{prop1} Let $I$ be a locally principal nonzero ideal of $R$. Then 
$I$ is invertible if and only if $(R:I)_M=(R_M:I_M)$, for each  $M\in \Max(R)$.
\end{proposition}
\begin{proof} If $I$ is invertible, it is finitely generated. Hence by flatness $(R:I)_M=(R_M:I_M)$. 
Conversely, if $I$ is locally principal and  $(R:I)_M=(R_M:I_M)$, we have:
$$I(R:I)=\bigcap_{M\in \Max(R)}I(R:I)R_M= \bigcap_{M\in \Max(R)}I_M(R_M:I_M)=\bigcap_{M\in \Max(R)}R_M=R.$$ \end{proof}

Following D.D. Anderson and M. Zafrullah, for short we call $R$ an \emph{LPI-domain} (or say that $R$ has the \emph{LPI-property}) if each locally principal nonzero ideal of $R$ is invertible  \cite{AZ}.

\begin{corollary} If $(R:I)_M=(R_M:I_M)$, for each nonzero ideal $I$ and each maximal ideal $M$ of $R$, then $R$ is an LPI-domain.
\end{corollary}

 If $\{R_\al\}$ is a family of overrings of $R$ such that $R= \bigcap_\al R_\al$, we say that the intersection $\bigcap_\al R_\al$ has \emph{finite character} if each nonzero element  of $R$ is not invertible  in  at most finitely many $R_\al$. If the intersection $R=\bigcap_{M\in \Max(R)}R_M$ has finite character, we say that $R$ has finite character.

An \emph{$h$-local} domain is a domain with finite character such that each nonzero prime ideal is contained in a unique maximal ideal. If $R$ is $h$-local,  we have $(X:Y)_M=(X_M:Y_M)$, for each couple of $R$-submodules $X$, $Y$ of $K$ such that $X\sub Y$ and each  $M\in \Max(R)$ \cite[Lemma 2.3]{BS2}. Thus by Proposition \ref{prop1} we get:

\begin{proposition} \cite[Lemma 3.9]{O3} An $h$-local domain is an LPI-domain.
\end{proposition}

Another sufficient condition for the equality  $(R:I)_M=(R_M:I_M)$ is that $I$ be $v$-finite. 

Recall that, if $I$ is a nonzero fractional ideal of a domain $R$, the \emph{divisorial closure} of $I$ is  the star operation defined by $I_v:=(R:(R:I))$. The ideal $I$ is called \emph{divisorial} if $I=I_v$ and a divisorial ideal is called \emph{$v$-finite} if $I=J_v$ for some finitely generated ideal $J$. Invertible ideals are divisorial. An introduction to star operations and divisorial closure is  \cite[Section 44]{gilmer}.

Now, observing that $(R:J_v)=(R:J)$, if $I=J_v$, with $J$ finitely generated, we have 
$$(R:I)_M \sub (R_M:I_M) \sub (R_M:J_M)=(R:J)_M=(R:I)_M.$$
Hence $(R:I)_M = (R_M:I_M)$.

\begin{proposition} \cite[Theorem 2.1]{A} \label{And1} Let $I$ be a nonzero  locally principal ideal of $R$.  Then $I$ is invertible if and only if $I$ is divisorial $v$-finite.
\end{proposition} 
\begin{proof} Any invertible ideal is divisorial  an finitely generated, hence $v$-finite. Conversely, if $I=J_v$, with $J$ finitely generated, as above we have $(R:I)_M = (R_M:I_M)$. Thus, if $I$ is also locally principal, $I$ is invertible by Proposition \ref{prop1}.
\end{proof}

A class of domains where each divisorial ideal is $v$-finite is the class of Mori domains.
 $R$ is called a \emph{Mori domain} if it satisfies the ascending chain condition on divisorial integral ideals. Clearly Noetherian domains and Krull domains are Mori. For the main properties of Mori domains the reader is referred to \cite{Bar}. 

From Proposition \ref{And1}, we immediately see that:

\begin{proposition} A Mori domain is an LPI-domain.
\end{proposition}

Since a local domain is trivially an LPI-domain,  each domain is intersection of LPI-domains. Thus, another way of approaching Problem \ref{Q1} is to investigate when an intersection of LPI-domains with the same quotient field  is an LPI-domain. It turns out that this is true if the intersection has finite character.

\begin{theorem} \cite[Theorem 4]{AZ} Let $\{R_\al\}$ be a family of overrings of $R$ such that $R=\bigcap R_\al$. If each $R_\al$ is an LPI-domain and the intersection has finite character, then $R$ is an LPI-domain.
\end{theorem}
\begin{proof} Let $I$ be a locally principal nonzero ideal and let $R_i:=R_{\al_i}$, $i=1, \dots, n$, be the overrings of the family containing $I$. Since $R_i$ is an LPI-domain, $IR_i$ is finitely generated, for each $i=1, \dots, n$. Thus, by the finite character, we can find a finitely generated ideal $J\sub I$ such that $JR_\al=IR_\al$, for each $\al$. By the properties of star operations, this implies that $I=J_v$. Thus $I$ is invertible by Proposition \ref{And1}.
\end{proof}

Now, let us introduce some more notation.
 The $t$-\emph{operation} is the star operation of finite type associated to the divisorial closure and it is therefore defined by setting 
$$I_t:=\bigcup\{J_v;   \mbox{ $J$ finitely generated and $J\sub I$}\},$$
for each nonzero ideal $I$ of $R$. If $I=I_t$, $I$ is called a \emph{$t$-ideal}. Note that $I_t\sub I_v$; thus each divisorial ideal is a $t$-ideal and $I_t=I_v$ if $I$ is finitely generated.

By Zorn's Lemma each $t$-ideal is contained in
a $t$-maximal ideal, which is prime, and we have 
$$R=\bigcap \{R_{M};\; M \mbox{ $t$-maximal ideal}\}.$$ 
We say that $R$ has \emph{$t$-finite character} if this intersection has finite character.

\begin{corollary} \label{FC} \cite[Corollary 1]{AZ}, \cite[Proposition 1.3]{PT}
A domain with finite character or $t$-finite character  is an LPI-domain.
\end{corollary}

Since Mori domains have $t$-finite character \cite[Theorem 3.3]{Bar},
from Corollary \ref{FC} we recover that, as seen above,  $h$-local domains and Mori domains have the LPI-property.

Observe that a locally principal nonzero  ideal that is contained just in finitely many maximal ideals need not be invertible. In fact, as already observed, a maximal ideal of an almost Dedekind domain is always locally principal but need not be finitely generated. A sufficient condition for a locally principal nonzero  ideal $I$ to be finitely generated is that $I$ contains an element belonging  to only finitely many maximal ideals \cite[Lemma 37.3]{gilmer}. In addition, if $I$ contains such an element, $I$ can be generated by two elements \cite[Corollary 1]{GH}.
This is a direct way of proving that the finite character implies the LPI-property. 

However,  the LPI-property does not imply neither the finite character nor the $t$-finite character. In fact any Noetherian domain is an LPI-domain, but need not have finite character.  On the other hand,  there are examples of local domains (hence, LPI-domains) with infinitely many $t$-maximal ideals with nonzero intersection (hence, without $t$-finite character)\cite[Example 1.13]{FPT}.
Thus we ask:

\begin{Qu} \label{Qt} When the  LPI-property implies the $t$-finite character?
\end{Qu}

\begin{Qu} \label{Q2} When the  LPI-property implies the finite character? \end{Qu}

To study Problem \ref{Qt}, it is natural to  generalize the LPI-property by means of the $t$-operation.

Recall that a nonzero ideal $I$ of $R$ is \emph{$t$-invertible} if there is an ideal $J$ such that $(IJ)_t=R$. A digression on $t$-invertibility can be found in \cite{Z}.

To our purpose, we recall that $t$-invertible ideals have a characterization similar to the one given in Proposition \ref{inv} for invertible ideals. That is, a $t$-ideal $I$ is $t$-invertible if and only if is $v$-finite (i.e., $I=J_v$ with $J$ finitely generated) and is $t$-locally principal (i.e., $IR_M$ is principal, for each $t$-maximal ideal $M$) \cite[Theorems 1.1 and 2.2]{Z}. 

We say that $R$ has the  \emph{$t$LPI-property} if each $t$-ideal which is $t$-locally principal is $t$-invertible. 
 The $t$LPI-property implies the LPI-property \cite[Proposition 2.2]{FPT}.
 
A \emph{Pr\"ufer domain} is a domain whose localizations at prime ideals are valuation domains, equivalently, a domain with the property that finitely generated  nonzero ideals are invertible \cite[Chapter IV]{gilmer}. The $t$-analog of Pr\"ufer domains are the so called \emph{Pr\"ufer $v$-multiplication domains}, for short P$v$MDs, defined by the property that each localization at a prime $t$-ideal is a valuation domain, equivalently, by the property that each finitely generated nonzero  ideal is $t$-invertible \cite{MZ}.

Since almost Dedekind domains are Pr\"ufer, not all Pr\"ufer domains are LPI-domains. S. Bazzoni conjectured that Pr\"ufer domains with the LPI-property were of finite character \cite[p. 630]{B1}. 
This conjecture was then proved by W. Holland, J. Martinez, W. McGovern and M. Tesemma in \cite{HM}. (A simplified proof is in \cite{McG}). F. Halter-Koch gave independently another proof, in the more general context of  ideal systems. In particular, he proved the analog of Bazzoni's conjecture for P$v$MDs, that is: a P$v$MD with the $t$LPI-property has $t$-finite character  \cite[Theorem 6.11]{HK}. With different methods, the same result was obtained by M. Zafrullah  \cite[Proposition 5]{Za}.

Other contributions to Problem \ref{Qt} were given by C.A. Finocchiaro, G. Picozza and F. Tartarone. In \cite{FPT}, they extended  several results of Halter-Koch and Zafrullah in the context of star operations of finite type and proved a suitable generalization of Bazzoni's conjecture for $v$-coherent domains. 
Recall that 
$R$ is called \emph{$v$-coherent} if, whenever  $I$ is finitely generated, $(R:I)$ is $v$-finite. The class of $v$-coherent domains properly include coherent domains, Mori domains and P$v$MDs \cite[Example 1]{N}, \cite{GaH}.

Since we do not know any example of $t$LPI-domain without $t$-finite character, we ask:

 \begin{Qu} \cite[p. 657]{Za}Ê\label{Qt2} Assume that $R$ is a $t$LPI-domain. Is it true that $R$ has $t$-finite character?
\end{Qu}

Besides that for P$v$MDs, we know that the answer is positive when each prime $t$-ideal of $R$ is contained in a unique $t$-maximal ideal \cite[Proposition 5.7]{GP1}.
 
Now we go back to the LPI-property. The following theorem gives useful characterizations of domains with finite character and of LPI-domains.

\begin{theorem} \label{teoFPT} 
\begin{enumerate}

\item[(1)]  \cite[Proposition 1.6]{FPT}, \cite[Corollary 4]{ZD} $R$ has finite character if and only if  any family of pairwise  comaximal  finitely generated ideals with nonzero intersection is finite.

\item[(2)]Ê\cite[Theorem 1]{Za}  If $R$ is an LPI-domain,  any family of pairwise  comaximal invertible ideals with nonzero intersection is finite.
\end{enumerate}
\end{theorem}

By means of this theorem, to answer Problem \ref{Q2} one can try the following strategy. Assume, by way of contradiction, that $R$ does not have finite character and consider an infinite family  of maximal ideals, or an infinite family of pairwise  comaximal  finitely generated ideals, with nonzero intersection. When, from one of these families, it is possible to recover an infinite family of pairwise  comaximal  invertible ideals with nonzero intersection, then $R$ cannot be an LPI-domain.

Since  finitely generated nonzero ideals of Pr\"ufer domains are invertible, in this way immediately we get:

\begin{corollary} [Bazzoni's Conjecture] A Pr\"ufer LPI-domain has finite character.
\end{corollary}

To go further, we enlarge the class of Pr\"ufer domains to finitely stable domains.

 Recall that an ideal $I$ of $R$ is  called \emph{stable} if $I$ is invertible in its endomorphism ring $E(I):=(I:I)$. Invertible ideals are stable.
The  domain $R$ is said to be (\emph{finitely}) \emph{stable} if each nonzero (finitely generated) ideal $I$ is stable. 
Stability was thoroughly investigated by B. Olberding  \cite{O3, O1, O2}. In particular, he proved that a domain $R$ is stable if and only if it has finite character and $R_M$ is stable for each $M\in \Max(R)$ \cite[Theorem 3.3]{O2}.

 Of course stability and finite stability coincide for Noetherian domains; but in general these two notions are distinct. In fact, an integrally closed finitely stable domain is precisely a  Pr\"ufer domain. (To see this, just recall that $R$ is integrally closed if and only if $R=(I:I)$ for each finitely generated ideal $I$.) On the other hand, a valuation domain $V$ is stable if and only if $PV_P$ is principal, for each nonzero prime ideal $P$ \cite[Proposition 4.1]{O3}.

This shows that a finitely stable domain need not have neither the LPI-property nor finite character. Also observe that, for finitely stable domains, finite character and $t$-finite character coincide.  
In fact, when $R$ is finitely stable, each maximal ideal is a $t$-ideal; this follows from \cite[Lemma 2.1 and  Theorem 2.4]{tlink} since a finitely stable domain has Pr\"ufer integral closure \cite[Proposition 2.1]{Rush}. 

Thus  one is lead to investigate whether, as in the Pr\"ufer case, a finitely stable LPI-domain has finite character. 

\begin{Qu} \cite[Question 4.6]{B4} \label{LPI} Assume that $R$ is a finitely stable LPI-domain. Is it true that $R$ has finite character?
\end{Qu}

By using Theorem \ref{teoFPT}, we are able to answer this question in positive under some additional hypotheses (see also \cite{G}). 
Observe that when $R$ is Mori, the answer is  trivially affirmative.  In fact Mori domains are LPI-domains (Proposition 7) and Mori finitely stable domains have always finite character, because  Mori domains have $t$-finite character \cite[Theorem 3.3]{Bar} and, as said before, the maximal ideals of a finitely stable domain are $t$-ideals. 

\begin{theorem} \label{LPI-FC} 
Let $R$ be a finitely stable LPI-domain verifying one of the following conditions:
\begin {itemize}
 
\item[(a)] Each nonzero finitely generated  ideal $I$ is principal in $E(I)$.

\item[(b)] Each maximal ideal of $R$ is stable.

\item[(c)] \cite[Lemma 3.9]{O3} Each nonzero prime ideal of $R$ is contained in a unique maximal ideal (e.g., $R$ is one-dimensional).

\item[(d)] Each fractional overring of $R$ is an LPI-domain.
\end{itemize}
Then $R$ has finite character.
\end{theorem}

\begin{proof} (a) Let $\{I_\al\}$ be a family of pairwise comaximal finitely generated ideals of $R$ with nonzero intersection. By hypothesis, for each fixed $\al$,  $I_\al=x_\al E(I_\al)$, $x_\al\in I_\al$, and so $I_\al^2=x_\al I_\al$. Hence $I_\al^2\sub x_\al R\sub I_\al$. It follows that the principal ideals $x_\al R$ are  pairwise comaximal.  If $y\in R$ is a nonzero element contained in each $I_\al$, then $y^2\in I_\al^2\sub x_\al R$, for each $\al$. By the LPI-property of $R$, the set of indexes $\al$ must be finite (Theorem \ref{teoFPT}(2)). Hence $R$ has finite character (Theorem \ref{teoFPT}(1)).

(b) Let $M$ be a maximal ideal of $R$. Since $M$ is stable,  $MR_M=x(MR_M:MR_M)$ for some $x\in M$ \cite[Lemma 3.1]{O2}. Then $M^2R_M\sub xMR_M\sub xR_M$. 
It follows that the ideal $I:=xR_M\cap R$, containing $M^2$, is not contained in any maximal ideal of $R$ different from $M$. Since $IR_M=xR_M$,  $I$ is locally principal and so it is invertible by the LPI-property. 

Now, suppose that $y$ is a nonzero element of $R$ contained in infinitely many maximal ideals $M_\al$. For each $\al$, consider the ideal $I_\al$ constructed in the preceding paragraph. Then $y^2\in M_\al^2\sub I_\al$, for each $\al$.
Since $\{I_\al\}$ is an infinite family of pairwise comaximal invertible ideals, this contradicts Theorem \ref{teoFPT}(2).

(c) Let $x$ be a nonzero element of $R$ contained in infinitely many maximal ideals $M_\al$ and consider the ideals $I_\al:=xR_{M_\al}\cap R$. By hypothesis, for a fixed $\al$ and maximal ideal $N$, if $N\neq M_\al$,   there is no nonzero prime ideal contained in the intersection $N\cap M_\al$. Thus $M_\al$ is the unique maximal  ideal containing $I_\al$. Since $I_\al R_{M_\al}=xR_{M_\al}$,  $I_\al$ is locally principal and so it is invertible  by the LPI-property. 
It follows that $\{I_\al\}$ is an infinite family of pairwise comaximal invertible ideals containing $x$, in contradiction with Theorem \ref{teoFPT}(2).

(d) Assume that $R$ does not have finite character. Then, by  Theorem \ref{teoFPT}(1), there exists a  nonzero element $x\in R$ which is contained in infinitely many pairwise comaximal finitely generated ideals $I_\al$. Set $E_\al:=(I_\al:I_\al)$ and consider the $R$-module $E:=\sum_\al E_\al$. Since $E$ is contained in the integral closure of $R$ and $R$ is finitely stable, $E$ is an overring of $R$ \cite[Proposition 2.1]{Rush}.
In addition, since $x\in I_\al$ for each $\al$, $xE\sub \sum_\al I_\al\sub R$ and so $E$ is a fractional overring of $R$. 

We claim that $\{I_\al E\}$ is an infinite  family of pairwise comaximal  invertible  ideals of $E$ containing $x$. Since, by hypothesis,  $E$ has the LPI-property, this contradicts  Theorem \ref{teoFPT}(2).

First of all, since $R$ is finitely stable, each ideal $I_\al$ is  stable, so that $I_\al E$ is  invertible in $E$.
Then, since  $E$ is integral over $R$, 
each maximal ideal of $R$ is contained in a maximal ideal of $E$. It follows that $I_\al E\neq E$ for each $\al$.  In addition, for $\al\neq \be$ the ideals $I_\al E$ and $I_\be E$ are comaximal, since the contraction of a prime ideal of $E$ is a prime ideal of $R$.
\end{proof}

A particular case of Theorem \ref{LPI-FC}(b) is when each maximal ideal of $R$ is invertible. In this case, the conclusion follows directly from Theorem \ref{teoFPT}(2). 
We remark that a domain whose prime ideals are all stable is not necessarily stable, even if it is finitely stable \cite[Section 3]{O1}.

By using the techniques developed in \cite{GP1}, it is possible to show that Theorem \ref{LPI-FC} holds in the more general context of star operations spectral and of finite type (see for example \cite[Proposition 5.5]{GP1} and the proof of \cite[Theorem 5.2]{GP1}).

S. Bazzoni noted  that the LPI-property  can be strengthened by the condition that each nonzero ideal $I$ of $R$ that is locally stable (i.e.,  $IR_M$ is stable, for each maximal ideal $M$) is stable. She called this condition the \emph{local stability property}. B. Olberding proved that a domain with finite character has the local stability property \cite[Lemma 4.3]{O3} and S. Bazzoni proved the converse \cite[Theorem 4.5]{B4} for finitely stable domains. We can recover her result from Theorem \ref{LPI-FC} in the following way.

\begin{lemma} \cite[Lemma 3.2]{B4} \label{LSP} A domain with the local stability property is an LPI-domain.
\end{lemma}
\begin{proof} Assume that $R$ has the local stability property. If $I$ is a nonzero locally principal ideal of $R$, then $I$ is locally stable. Hence $I$ is stable. In particular $I$ is finitely generated in $E(I)$ and so $E(I)_M=(I:I)_M=(I_M:I_M)=R_M$, for each maximal ideal $M$. It follows that $E(I)=R$. Hence $I$ is invertible.
\end{proof}

\begin{proposition} \label{LSover} Let $R$ be a finitely stable domain. If $R$ has the local stability property, each fractional overring of $R$ has the local stability property.
\end{proposition}
\begin{proof} Let  $E$ be a fractional overring of $R$. Given a maximal ideal $M$ of $R$, $R_M$ is finitely stable and so $E_M$ is semilocal (i.e., it has finitely many maximal ideals); in fact any overring of a finitely stable domain is finitely stable and any overring of a local finitely stable domain is semilocal \cite[Lemma 2.4 and Corollary 2.5]{O2}. 
It follows that if $I$ is a locally stable ideal of $E$, then $IE_M$ is stable for each $M\in \Max(R)$. Hence, by the local stability property of $R$, the ideal $I=IE$ is stable.
\end{proof}

\begin{corollary} \label{corFC} 
Let $R$ be a finitely stable domain. The following conditions are equivalent:
\begin{itemize}
\item[(i)] $R$ has finite character;
\item[(ii)]  $R$ has the local stability property;
\item[(iii)] Each fractional overring of $R$ has the local stability property;
\item[(iv)] Each fractional overring of $R$ is an LPI-domain.
\end{itemize}
\end{corollary}

\begin{proof} (i) $\ra$ (ii) by \cite[Lemma 4.3]{O3}. (ii) $\ra$ (iii) by Proposition \ref{LSover}. 
(iii) $\ra$ (iv) by Lemma \ref{LSP}. (iv) $\ra$ (i) by Theorem \ref{LPI-FC}(d). 
\end{proof}

 \begin{corollary} \cite[Theorem 4.5]{B4} A finitely stable domain with the local stability property has finite character. 
 \end{corollary}

By Corollary \ref{corFC}, we see that Problem \ref{LPI} can be reformulated in the following way:

\begin{Qu}  \label{QuLSI}ÊIs it true that every fractional overring of a finitely stable LPI-domain is an LPI-domain?
\end{Qu}

Note that the local stability property and the LPI-property coincide for Pr\"ufer domains \cite[Proposition 3.3]{B4} and that the answer to Problem \ref{QuLSI} (equivalently, to Problem \ref{LPI}) is positive if and only if these two properties are more generally equivalent for finitely stable domains \cite[Question 6]{B4}. 


\end{document}